\documentclass[12pt]{amsart}
\usepackage{amsthm}
\usepackage{amsmath}
\usepackage{amssymb}
\usepackage{mathrsfs}
\usepackage{enumerate}
\usepackage{float}
\usepackage{url, multicol}
\usepackage[all]{xy}
\usepackage{graphicx, pgf}
\usepackage[left = 1 in,top = 1 in, bottom=1in,right=1in]{geometry}
\usepackage{hyperref}
\usepackage{appendix}
\definecolor{green}{rgb}{0.0, 0.5, 0.0}
\definecolor{purple}{rgb}{0.5, 0.0, 0.5}
\definecolor{bluegreen}{rgb}{0.0,0.5, 0.5}
\definecolor{orange}{rgb}{1,0.5, 0.1}
\definecolor{redgreen}{rgb}{0.5, 0.5, 0.0}

\newcommand{\Q}{{\mathbb Q}}
\newcommand{\C}{{\mathbb C}}
\newcommand{\Z}{{\mathbb Z}}

\newcommand{\R}{{\mathbb R}}

\renewcommand{\P}{{\mathbb P}}

\newcommand{\OO}{{\mathcal O}}
\renewcommand{\b}{{\mathbb B}}

\renewcommand{\a}{\mathrm{a}}
\renewcommand{\b}{\mathrm{b}}
\renewcommand{\c}{\mathrm{c}}

\newcommand{\rank}{\mathrm{rank}}
\newcommand{\ord}{\mathrm{ord}}

\title{Rational Hyperbolic Triangles and a Quartic Model of Elliptic Curves}
\author{Nicolas Brody and Jordan Schettler}

\begin{document}

\allowdisplaybreaks

 \setcounter{tocdepth}{1}
\newtheorem{thm}{Theorem}
\newtheorem{conj}[thm]{Conjecture}
\newtheorem{prop}[thm]{Proposition}
\newtheorem{lemma}[thm]{Lemma}
\newtheorem{corollary}[thm]{Corollary}
\newtheorem*{fact}{Fact}
\theoremstyle{remark}
\newtheorem{rem}[thm]{Remark}
\theoremstyle{definition}
\newtheorem{defn}[thm]{Definition}
\newtheorem*{goal}{Goal}
\newtheorem{assume}{Assumption}
\renewcommand{\theassume}{\Alph{assume}}
\newtheorem{exam}[thm]{Example}
\numberwithin{equation}{section}
%\numberwithin{equation}{thm}
%\numberwithin{table}{thm}

\begin{abstract}
%In this short note, we construct and study two families of elliptic curves, both of which relate the group structures of the curves to certain Euclidean or hyperbolic triangles. We first re-construct a family of nonsingular plane cubic curves which parameterize Euclidean triangles with fixed perimeter and incircle; moreover, rational triangles correspond to rational points. Then we derive an analogous construction for hyperbolic triangles. This leads to a family of singular quartic plane curves. There is an appropriate notion of rational hyperbolic triangles and, again, rational triangles correspond to rational points.
The family of Euclidean triangles having some fixed perimeter and area can be identified with a subset of points on a nonsingular cubic plane curve, i.e., an elliptic curve; furthermore, if the perimeter and the square of the area are rational, then the curve has rational coordinates and those triangles with rational side lengths correspond to rational points on the curve. We first recall this connection, and then we develop hyperbolic analogs. There are interesting relationships between the arithmetic on the elliptic curve (rank and torsion) and the family of triangles living on it. In the hyperbolic setting, the analogous plane curve is a quartic with two singularities at infinity, so the genus is still $1$. We can add points geometrically by realizing the quartic as the intersection of two quadric surfaces. This allows us to construct nontrivial examples of rational hyperbolic triangles having the same inradius and perimeter as a given rational right hyperbolic triangle.
\end{abstract}

\maketitle

\section{Introduction}

Connections between families of triangles and elliptic curves have been long studied.
Consider, for example, the famous congruent number problem: What rational numbers $A$
are the area of a right triangle with rational side lengths?
We call such $A$ \emph{congruent numbers}.
Using similar triangles, we may assume $A$ is a squarefree positive integer. For example, 6 is a congruent number since it is the area of
the $(3,4,5)$-right triangle. Euler showed that $7$ is a congruent number, but
$1$, $2$, and $3$, are not. It turns out that a %squarefree positive integer $A$ is a congruent number
%if and only if there is a rational point on the elliptic curve $y^2 = x^3 - A^2x$ whose
%$x$-coordinate is the square of a rational number and has even denominator.
positive rational number $A$ is a congruent number if and only if the elliptic curve $y^2 = x^3 - A^2x$
has a rational point in the first quadrant; in fact, since the the only torsion
points on this elliptic curve have order dividing two, $A$ being congruent here is
equivalent to the group of rational points on the curve having positive rank.
More generally, one can ask whether $A$
is the area of a \emph{rational triangle} (i.e., a triangle having rational side lengths) one of whose internal angles is some fixed value $\theta \in(0,\pi)$. Such numbers 
correspond to the existence of a certain kind of rational point on elliptic curves of the form $y^2 = x(x-A\lambda)(x+A/\lambda)$
where $\lambda = \sin(\theta)/(\cos(\theta) +1)$ as found in Problem 3, Section 2, Chapter I of \cite{Kobl}.

Similarly, one can consider rational numbers $A$ which are the area of a rational triangle having some
fixed perimeter instead of a fixed internal angle. %, but note that we seemingly cannot exploit similar triangles here.
Again, the existence of such a triangle corresponds to the existence of a rational point on an elliptic curve,
namely, a rational point in the first quadrant
on the curve $s^2xy-A^2=sxy(x+y)$ where $s$ is the \emph{semiperimeter} defined to be half the perimeter.
We will recall this connection in Section \ref{sectone}.
Note that the area $A$ of a triangle is determined via Heron's formula $A = rs$ where $r$ is the inradius, so we are led
to study rational points on the curves $C_{r,s} \colon s(xy-r^2)=xy(x+y)$, which are
similar to those studied in \cite{Camp} and/or \cite{Goin}. From the equation for $C_{r,s}$, it
is clear that similar triangles do, in fact, define isomorphic curves and that the family of curves
can be parameterized by $k = s/r$. We will keep the parameters $r$, $s$ separate, however,
to motivate the hyperbolic analogs which have no chance of exploiting similarity since
similar hyperbolic triangles are actually congruent.

In general, any triangle gives rise to a point on $C_{r,s}$ in the first quadrant where
$r$ is the triangle's inradius and $s$ is its semiperimeter. (Here $s\geq 3\sqrt{3}r >0$, and conversely, for any pair of real numbers $r,s$ satisfying this
inequality, there is a triangle with these parameters.) Such a curve $C_{r,s}$ will be an elliptic curve provided the
triangle we started with was not equilateral or, equivalently, $s>3\sqrt{3}r$. Moreover, if the given triangle is rational, then
the point will have rational coordinates and $C_{r,s}$ will have rational coefficients since here $A^2\in \Q$
even though $A$ itself is not necessarily rational.

In particular, when $s,r^2\in \Q$ and $s> 3\sqrt{3}r >0$, one can ask questions about how the structure of the Mordell-Weil group
$C_{r,s}(\Q)$ is related to the family of rational triangles with parameters $r,s$. For example, ``When do the points coming from rational triangles have infinite order?''
It turns out that many Pythagorean triples give rise to rational points of infinite order (see Theorem \ref{rankthm} below), a fact
which comes from the observation that points corresponding to triangles cannot have odd order. There can be points of even order
coming from triangles, and such a point will have order $2$ or $6$ if and only if the corresponding triangle is isosceles.

We can also play the same game for hyperbolic triangles. The natural curve we get for hyperbolic triangles of a fixed inradius $r$ and fixed semiperimeter $s$, however, is a quartic curve
\begin{align*}
Q_{\rho,\sigma}\colon \sigma(x^2y^2+xy + \rho^2(x^2 + xy + y^2-1)) = (1+\rho^2)(x^2y+xy^2)
\end{align*}
where $\rho = \sinh(r)$ and $\sigma=\tanh(s)$. Again, this will be an elliptic curve in the non-equilateral case (i.e., $\sigma > 3\sqrt{3}\rho\frac{1+\rho^2}{1+9\rho^2}$). Also, certain right triangles generate rational points on curves with rational coefficients. In fact, there is a hyperbolic Pythagorean theorem: for a right hyperbolic triangle
with legs $\ell_1$, $\ell_2$ and hypotenuse $\ell_3$, we have
\begin{align*}
\tilde{\ell}_3^2 = \tilde{\ell}_1^2+\tilde{\ell}_2^2-\tilde{\ell}_1^2\tilde{\ell}_2^2
\end{align*}
where, in general, $\tilde{u} = \tanh(u)$ for $u\in \R$. This motivates our definition of a
\emph{rational hyperbolic triangle} to be one whose side lengths $\ell_1$, $\ell_2$, $\ell_3$
have rational hyperbolic tangents $\tilde{\ell}_1$, $\tilde{\ell}_2$, $\tilde{\ell}_3$. Note, this notion
of a rational hyperbolic triangle is more general than that used by Hartshorne
and van Luijk in \cite{Hart2} where the authors consider the more restrictive condition that the side lengths are natural logs of rational numbers.
In contrast, our definition of a rational hyperbolic triangle is equivalent to requiring its side lengths to be natural logs of square roots of
rational numbers. In particular, the examples of hyperbolic Pythagorean triples derived by Hartshorne
and van Luijk are rational hyperbolic triangles under our definition, and in this case both $\sigma$ and $\rho^2$ are rational.

%Hartshorne's identity does not look as analogous as ours with the
%change of variable
%\begin{align*}
%1-\tilde{\ell}_1^2 = \left(\frac{2a}{a^2+1}\right)^2
%\end{align*}
%and this change change corresponds to $\tilde{\ell}_1 = \tanh(\ell_1)$ while $a = e^{\ell_1}$.

\section{Euclidean Triangles and Cubic Curves}\label{sectone}

Consider a Euclidean triangle $\Delta$. The angle bisectors are concurrent in a point called the incenter. Furthermore, the distances from the sides of $\Delta$ to the incenter are all equal, so this determines an inscribed circle, called the incircle. We define the inradius $r$ to be the radius of the incircle. The inradii joining the incenter to the three intersection points of the incircle with $\Delta$ determine three central angles $\alpha$, $\beta$, $\gamma$, and three partial side lengths $\a$, $\b$, $\c$, as in Figure \ref{E} below.
\begin{figure}
\begin{center}
\includegraphics[scale=0.23]{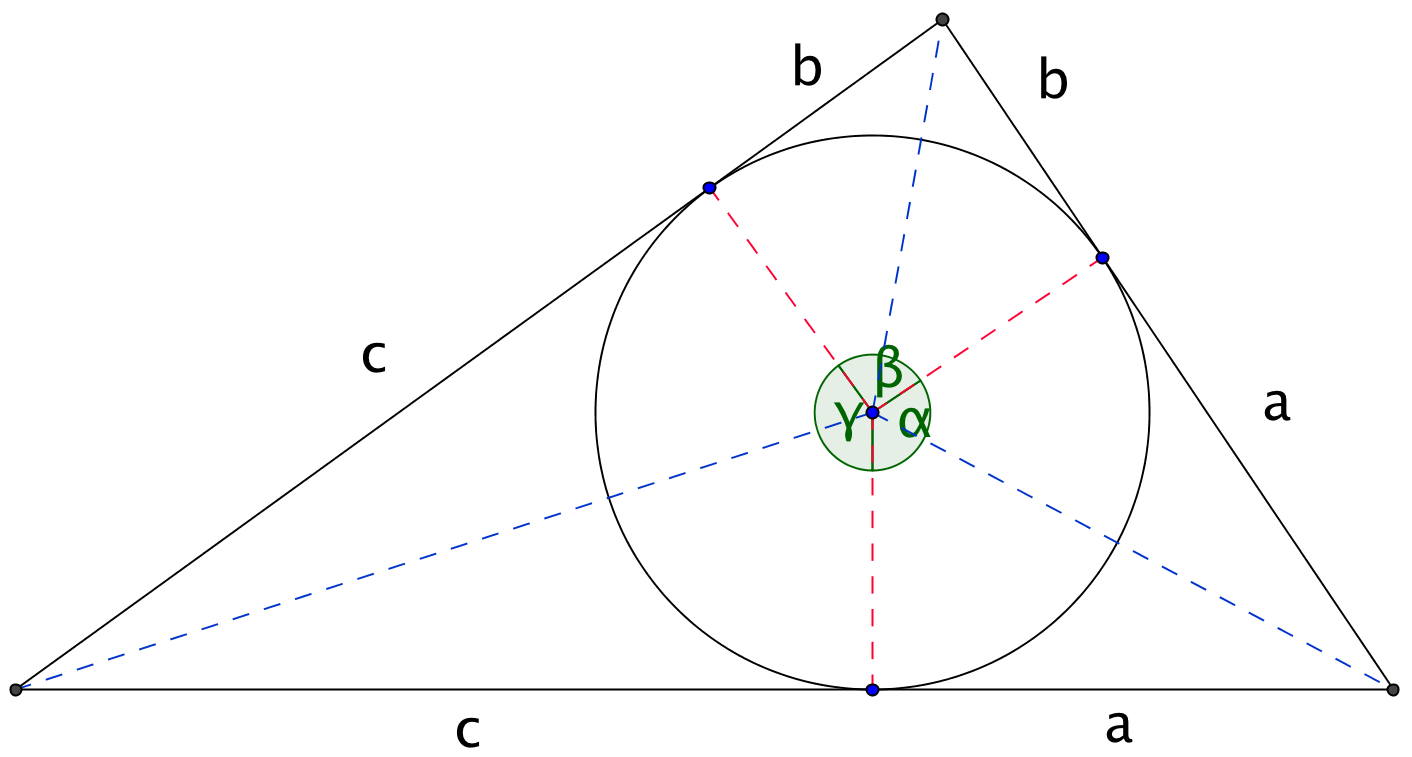}
\end{center}
\caption{Euclidean Triangle with Angle Bisectors, Incircle, and Inradii}
\label{E}
\end{figure}
The angle bisectors also bisect the central angles, so
\begin{align*}
\a = r\tan(\alpha/2) && \b = r\tan(\beta/2) && \c = r\tan(\gamma/2).
\end{align*}
Hence the Triple Tangent Identity (see \ref{trip} in Appendix \ref{A}) implies
\begin{align*}
\frac{\a\b\c}{\a+\b+\c} = r^2,
\end{align*}
so
\begin{align*}
\c = \frac{r^2(\a+\b)}{\a\b-r^2}.
\end{align*}
Thus the semiperimeter $s$ (half the perimeter) of $\Delta$ satisfies
\begin{align*}
s = \a + \b + \c = \a + \b +  \frac{r^2(\a+\b)}{\a\b-r^2},
\end{align*}
or equivalently
\begin{align}\label{cubic}
s(\a\b-r^2) = \a^2\b+\a\b^2.
\end{align}
One can show that $s\geq 3\sqrt{3}r$ with equality only for an equilateral triangle. Suppose $\Delta$ has side lengths $\ell_1,\ell_2,\ell_3$. Then, without loss of generality,
\begin{align*}
\a + \b=\ell_1  &&  \b+ \c = \ell_2 &&  \c+ \a = \ell_3,
\end{align*}
so
\begin{align}\label{one}
\a = \frac{\ell_1-\ell_2+\ell_3}{2} && \b = \frac{\ell_1+\ell_2-\ell_3}{2} && \c = \frac{-\ell_1+\ell_2+\ell_3}{2}.
\end{align}
Therefore Eq. \ref{cubic} implies that we get a point
\begin{align*}
(\a,\b) = \left(s-\ell_2, s-\ell_3\right)
\end{align*}
on the cubic curve %(see Figure \ref{crs})
\begin{align}\label{eq}
C_{r,s}\colon s(xy-r^2) = x^2y+xy^2 
\end{align}
where
\begin{align}\label{two}
s=\frac{\ell_1+\ell_2+\ell_3}{2}, && r^2 = \frac{(s-\ell_1)(s-\ell_2)(s-\ell_3)}{s}.
\end{align}
The graph of $C_{r,s}$ in the affine real plane has four components (see Figure \ref{crs}), one in each quadrant.
The point $(\a,\b)$ from $\Delta$ is, of course, in the first quadrant. %, and, conversely, any point in the
%first quadrant comes from a triangle.
\begin{figure}
\begin{center}
\includegraphics[width=12 cm]{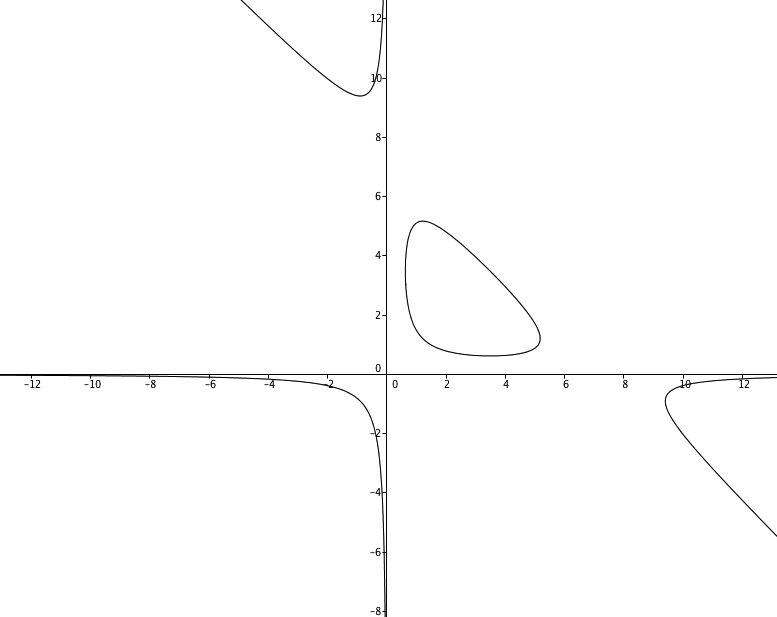}
\end{center}
\caption{Graph of $C_{r,s}$ in the affine real plane}
\label{crs}
\end{figure}
%$m>n>0$, $(m,n)=1$, $hyp = m^2+n^2$, $legs are 2mn, m^2-n^2$
%$$r = n(m-n)$$
%$$s = m(m+n)$$
%$$\mbox{Disc} = m^4n^6(m^2-n^2)^6(m^2-27n^2)$$
%$$ \frac{s^2(s^2-24r^2)^3}{r^6(s^2-27r^2)} = \frac{m^2(m+n)^2 ( )^3}{ n^6(m-n)^6(m^2-27n^2)}$$
%primes dividing area (or equiv legs) are of bad reduction

In fact, we get up to six points on $C_{r,s}$ coming from $\Delta$ (see Figure \ref{sixpoints}), namely, $(\a,\b)$, $(\b,\c)$, $(\c,\a)$ and the transposes $(\b,\a)$, $(\c,\b)$, $(\a,\c)$.
\begin{figure}
\begin{center}
\includegraphics[width=6.5 cm]{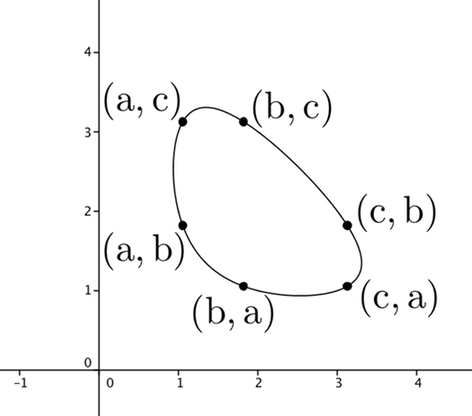}
\end{center}
\caption{Six points on $C_{r,s}$ coming from a triangle}
\label{sixpoints}
\end{figure}
We homogenize Eq. \ref{eq} to view $C_{r,s}$ as a curve in $\P^2(\C)$ with projective coordinates $[x,y,z]$; it is nonsingular provided that $s>3\sqrt{3}r$ (i.e., $\Delta$ is non-equilateral) and has the following three points at infinity: $[1,0,0]$, $[0,1,0]$, and $[1,-1,0]$. There is a natural group law on the points of $C_{r,s}$ which can be interpreted geometrically via the usual tangent/chord construction: for points $P$, $Q$, we take $P+Q$ to be the reflection about the line $y=x$ of the third intersection point $P\ast Q$ of $C_{r,s}$ with the secant/tangent line for $P$, $Q$. In particular, $[1,-1,0]$ is the identity element and $[1,0,0]$, $[0,1,0]$ both have order three with $[1,0,0]+[0,1,0]=[1,-1,0]$. For the affine point $(\a,\b)$ as above, we have $-(\a,\b) = (\b,\a)$, $(\a,\b)+[1,0,0] = (\b,\c)$ and similarly $(\a,\b)+[0,1,0] = (\c, \a)$. If $\Delta$ is isosceles (e.g., when $\a=\b$), then we get three corresponding points on $C_{r,s}$, one of which has order two.
%\begin{figure}
%\begin{center}
%\includegraphics[width=8 cm]{threepoints.png}
%\end{center}
%\caption{Three point on $C_{r,s}$ coming from an isosceles triangle}
%\end{figure}
In fact, there are always three real points on $C_{r,s}$ of order two when $s>3\sqrt{3}r$, and exactly two of those points come from distinct isosceles triangles.

There is an easily determined Weierstrass form which gives us a different, but related, structure of an elliptic curve in the non-equilateral case, namely,
\begin{align*}
E_{r,s}\colon Y^2 + sXY+sr^2Y = X^3.
\end{align*}
Explicitly,
\begin{align*}
(X,Y) = \left(-\frac{sr^2}{x}, \frac{sr^2y}{x} \right),
\end{align*}
which we can view as a map in projective coordinates
\begin{align*}
C_{r,s} \longrightarrow E_{r,s} \colon [x,y,z] \mapsto [-sr^2z, sr^2y, x].
\end{align*}
This transformation maps the points of order three $[0,1,0]$ and $[1,0,0]$ for $C_{r,s}$ to the identity $[0,1,0]$ and the point of order three $(0,0)$, respectively, for $E_{r,s}$ and maps the identity $[1,-1,0]$ for $C_{r,s}$ to the point of order three $(0, -sr^2) = -(0,0)$ in $E_{r,s}$. The curves $C_{r,s}$ are very natural and symmetric, but the curves $E_{r,s}$ are better suited for computer computations because of their simple Weierstrass equations.

We say that $\Delta$ is \emph{rational} if $\ell_1,\ell_2,\ell_3\in\Q$. In this case, Eq.s \ref{one} and \ref{two} imply $\a,\b,\c,r^2\in\Q$, so both $C_{r,s}$ and $E_{r,s}$ have coefficients in $\Q$ and $(\a,\b)$ is rational point on $C_{r,s}$ which maps to a rational point $(-sr^2/\a, sr^2\b/\a)$ on $E_{r,s}$; thus we can exploit the group structures on $C_{r,s}(\Q)$ and $E_{r,s}(\Q)$ to generate other triangles with rational side lengths having the same semiperimeter $s$ and inradius $r$.
\begin{thm}
Suppose $s,r^2\in\Q$ with $s\geq 3\sqrt{3}r > 0$. Then the assignment
\begin{align*}
(x_0,y_0) \mapsto (x_0+y_0, s-x_0, s-y_0)
\end{align*}
is a bijection from the rational points $(x_0,y_0)$ on $C_{r,s}$ having $x_0, y_0 > 0$ to the sequences of side lengths of rational Euclidean triangles having semiperimeter $s$ and inradius $r$.
%Suppose $\Delta$ is a rational Euclidean triangle. Then the semiperimeter $s$ is rational, so Eq.s \ref{one}, \ref{two} imply that the partial side lengths $\a,\b,\c$ and the squared inradius $r^2$ are all rational. Thus both $C_{r,s}$ and $E_{r,s}$ have coefficients in $\Q$ and $(\a,\b)$ maps to a nontrivial point in $E_{r,s}(\Q)$. Conversely, any point $(\a_0, \b_0)$ on $C_{r,s}$ with $\a_0, \b_0 > 0$ which maps to point in $E_{r,s}(\Q)$ corresponds to a rational Euclidean triangle $\Delta_0$ having semiperimeter $s$, inradius $r$, and side lengths $\a_0 + \b_0$, $s - \a_0$, $s - \b_0$.
\end{thm}
\begin{proof}
We have already seen from the above construction that sequences of side lengths of rational triangles give rise to rational points on $C_{r,s}$ in the first quadrant; in particular, the map is onto. Conversely, suppose $(x_0,y_0)$ is a rational point on $C_{r,s}$ in the first quadrant. Then we claim $(x_0+y_0, s-x_0, s-y_0)$ is a sequence of side lengths of a triangle. It suffices to check the inequalities
\begin{align*}
(x_0+y_0) + (s-x_0) = s+y_0 > s-y_0 \\
(x_0+y_0) + (s-y_0) = s+x_0 > s-x_0 \\
(s-x_0)+(s-y_0) = 2s - (x_0+y_0) > x_0 + y_0 
\end{align*}
where the last inequality follows from
\begin{align*}
\frac{s}{x_0+y_0} = \frac{x_0y_0}{x_0y_0-r^2} > 1.
\end{align*}
It remains only to note that $(x_0+y_0, s-x_0, s-y_0) = (x_0'+y_0', s-x_0', s-y_0') \Rightarrow (x_0, y_0) =(x_0', y_0')$.
\end{proof}
In light of the above result, we say that a point on $C_{r,s}$ is a \emph{triangle point} if it lies in the first quadrant of the affine real plane.
\begin{thm}
The sum of two triangle points on $C_{r.s}$ is not a triangle point, and the sum of a triangle point with a non-triangle point is a triangle point. It follows that the sum
of an odd (resp. even) number triangle points is (resp. is not) a triangle point.
\end{thm}
\begin{proof}
This follows from the observation that the component of $C_{r,s}$ in the first quadrant of the affine real plane is the boundary of a convex region.
In particular, any line intersecting this component will intersect in exactly two points counting multiplicity.
\end{proof}
Since the identity $[1,-1,0]$ on $C_{r,s}$ is not a triangle point, we immediately obtain the following.
\begin{corollary}\label{odd}
A triangle point on $C_{r,s}$ cannot have odd order.
\end{corollary}
By a \emph{Pythagorean triple} we mean a set of 3 integers which together form the side
lengths of a right triangle. A Pythagorean triple $\ell_1$, $\ell_2$, $\ell_3$ is said to be \emph{primitive} if
there is no prime which divides every $\ell_i$. Every primitive Pythagorean triple is of the form $m^2-n^2$, $2mn$, $m^2+n^2$
where $m>n$ are relatively prime positive integers, exactly one of which is even; the inradius and semiperimeter of such a Pythagorean triple are given by $r = n(m-n)$ and $s = m(m+n)$, respectively.
\begin{thm}\label{rankthm}
Let $m>n$ be relatively prime positive integers. Suppose that $m$ is neither an odd perfect cube nor twice a perfect cube.
Then the curve $\mathcal{C}_{m,n} \mathrel{\mathop :}= C_{n(m-n), m(m+n)}$ corresponding to
the Pythagorean triple $m^2-n^2$, $2mn$, $m^2+n^2$ has the property
\begin{align*}
\rank_{\Z} (\mathcal{C}_{m,n}(\Q)) \geq 1.
\end{align*}
\end{thm}
\begin{proof}
By Corollary \ref{odd}, it suffices to show that the condition on $m$ implies $\mathcal{C}_{m,n}$ cannot have a rational point of order
two. We do this by contradiction. Suppose $P$ is a rational of order two on $\mathcal{C}_{m,n}$. Then $P = (x,x)$
for some $x\in \Q$, so $s(x^2-r^2) = 2x^3$. By the rational root test, $x = t/2$ for some $t\in\Z$. Hence $t^3 = s(t+2r)(t-2r)$. Let $p$ be an odd prime dividing $m$.
Then $p$ also divides $t^3$, so the exact power of $p$ dividing $t^3$ is $3k$ for some positive integer $k$. Note that
both $r = n(m-n)$ and $s/m= m+n$ are relatively prime to $m$, so the exact power of $p$ dividing $m$ is also $3k$. Since $m$
is not an odd perfect cube, this forces $m=2^ij^3$ where $j$ is an odd integer and $i$ is a positive integer. On the other hand, $r$ is odd and $t$ is even since $m$ is even, so $3\cdot \ord_2(t) = i + \ord_2(t+2r) + \ord_2(t-2r) = i+2$. Therefore $i\equiv -2 \equiv 1 \pmod{3}$,
which contradicts the assumption that $m$ is not twice a perfect cube.
\end{proof}
\begin{exam}
Consider the familiar right triangle with side lengths $3$, $4$, $5$. The semiperimeter is $s = (3+4+5)/2 = 6$ and the inradius is $r=(6-3)(6-4)(6-5)/6=1$. We can take $\a = (3+4-5)/2 = 1$ and $\b = (-3+4+5)/2= 3$, which yields a point $(1,3)$ on $C_{1,6}$. This point maps to the point $P=(-6/1, 6\cdot 3/1) = (-6, 18)$ on the elliptic curve $E_{1,6}$. We can compute $3P = (-35/9, 343/27)$ in $E_{1,6}(\Q)$ either by hand or with SAGE, which comes from the point $(x_0,y_0)=(54/35, 49/15)$ on $C_{1,6}$. Thus we can construct a rational triangle with lengths $x_0+ y_0 = 101/21$,  $s-x_0 = 156/35$, and $s-y_0 = 41/15$. This produces a non-right triangle which also has inradius $1$ and semiperimeter $6$. In particular, we have two non-congruent triangles with rational side lengths having the same perimeter $12$ and area $6$. We should note that it is not \emph{a priori} obvious why rational Euclidean triangles having the same area and perimeter are not necessarily congruent (see \cite{Rose} for further discussion of this point).
\begin{figure}
\includegraphics[height=10cm]{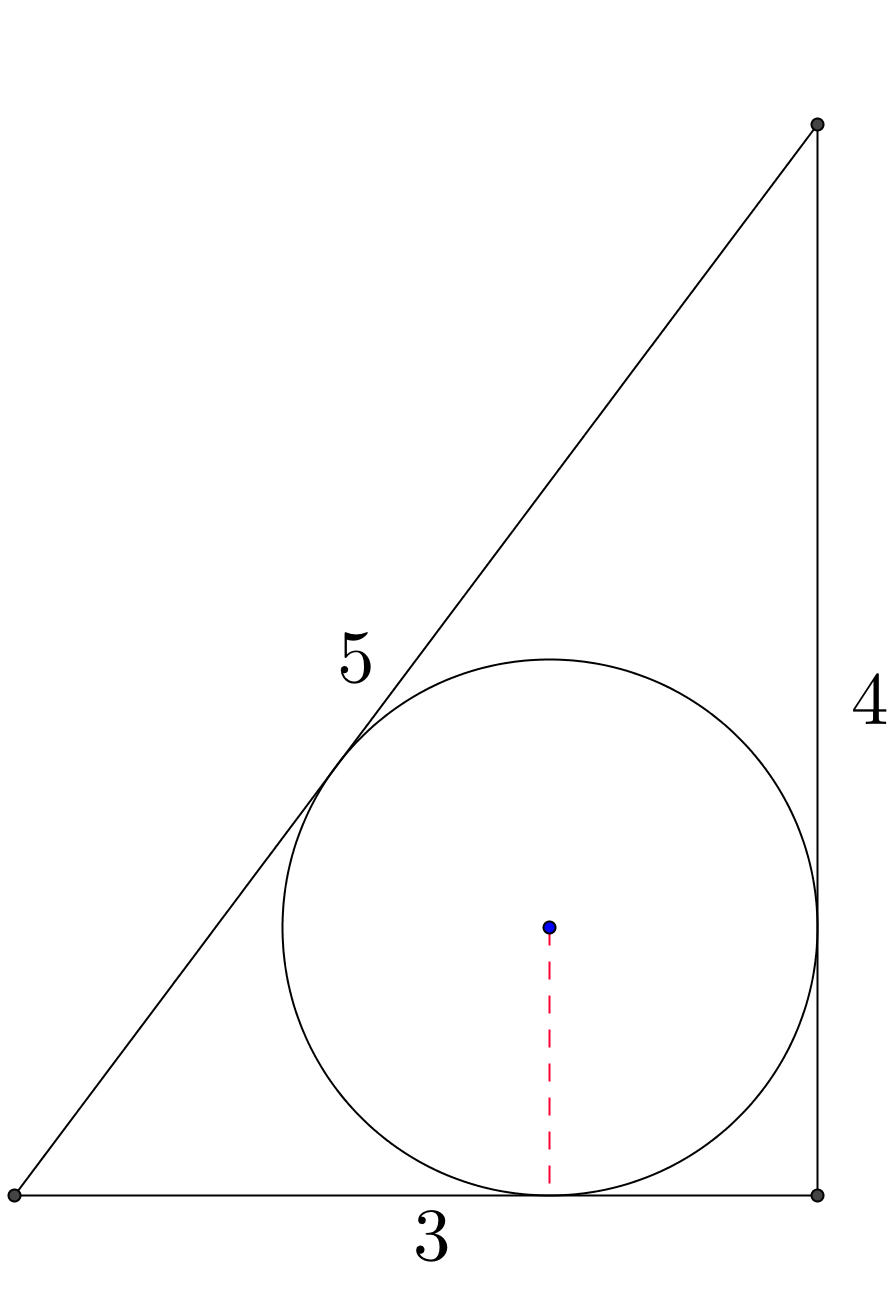}
\hspace{.5 in}
\includegraphics[height=10cm]{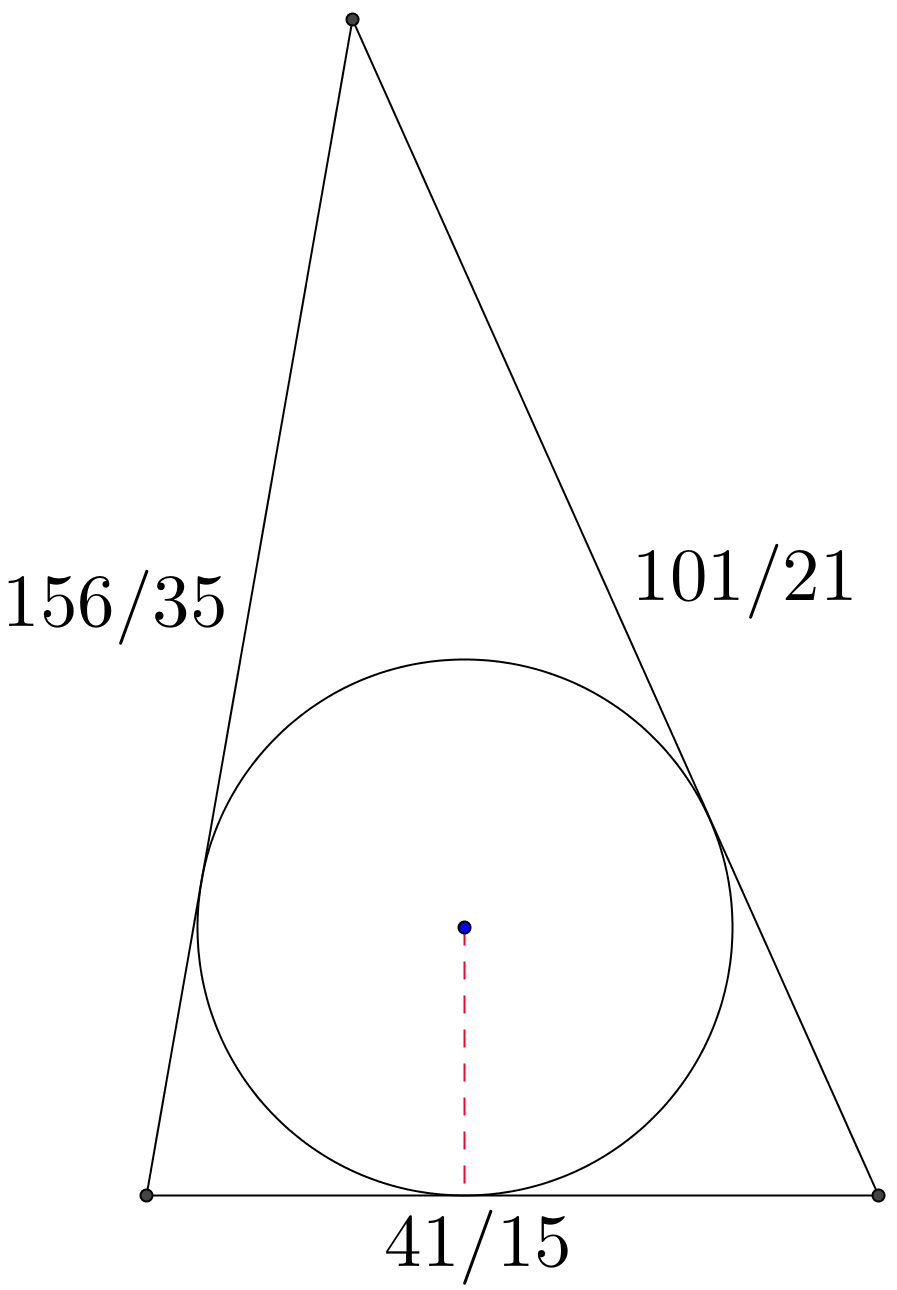}
\end{figure}
\noindent In fact, $E_{r,s}(\Q)$ has rank $1$ with $E_{r,s}(\Q) = \Z\cdot P$ as checked by SAGE, and we can generate arbitrarily many such examples in the same way by taking odd multiples of $P$. %The even multiples of $P$ do not give us triangles since the corresponding points on $C_{r,s}$ are not in the first quadrant. This is because the component of $C_{r,s}$ in the first quadrant is the boundary of a convex region, so the sum of two points on that component lies on a different component.
\end{exam}

\section{Hyperbolic Triangles and Quartic Curves}

Consider now a hyperbolic triangle $\Delta$ in the Poincar\'{e} upper half-plane, so each side is either a vertical line segment or the arc of a circle whose center lies on the boundary axis. Again, the angle bisectors are concurrent in a point called the incenter, and the hyperbolic distances from the sides of the triangle to the incenter are all equal, so this determines an inscribed circle, called the incircle. Note that a hyperbolic circle is a Euclidean circle, but the hyperbolic center is not the Euclidean center. We define the inradius $r$ to be the hyperbolic radius of the incircle. The inradii joining the incenter to the three intersection points of the incircle with $\Delta$ determine three central angles $\alpha$, $\beta$, $\gamma$, and three partial side lengths $\a$, $\b$, $\c$, as in Figure \ref{H} below.

\begin{figure}
\begin{center}
\includegraphics[scale=0.25]{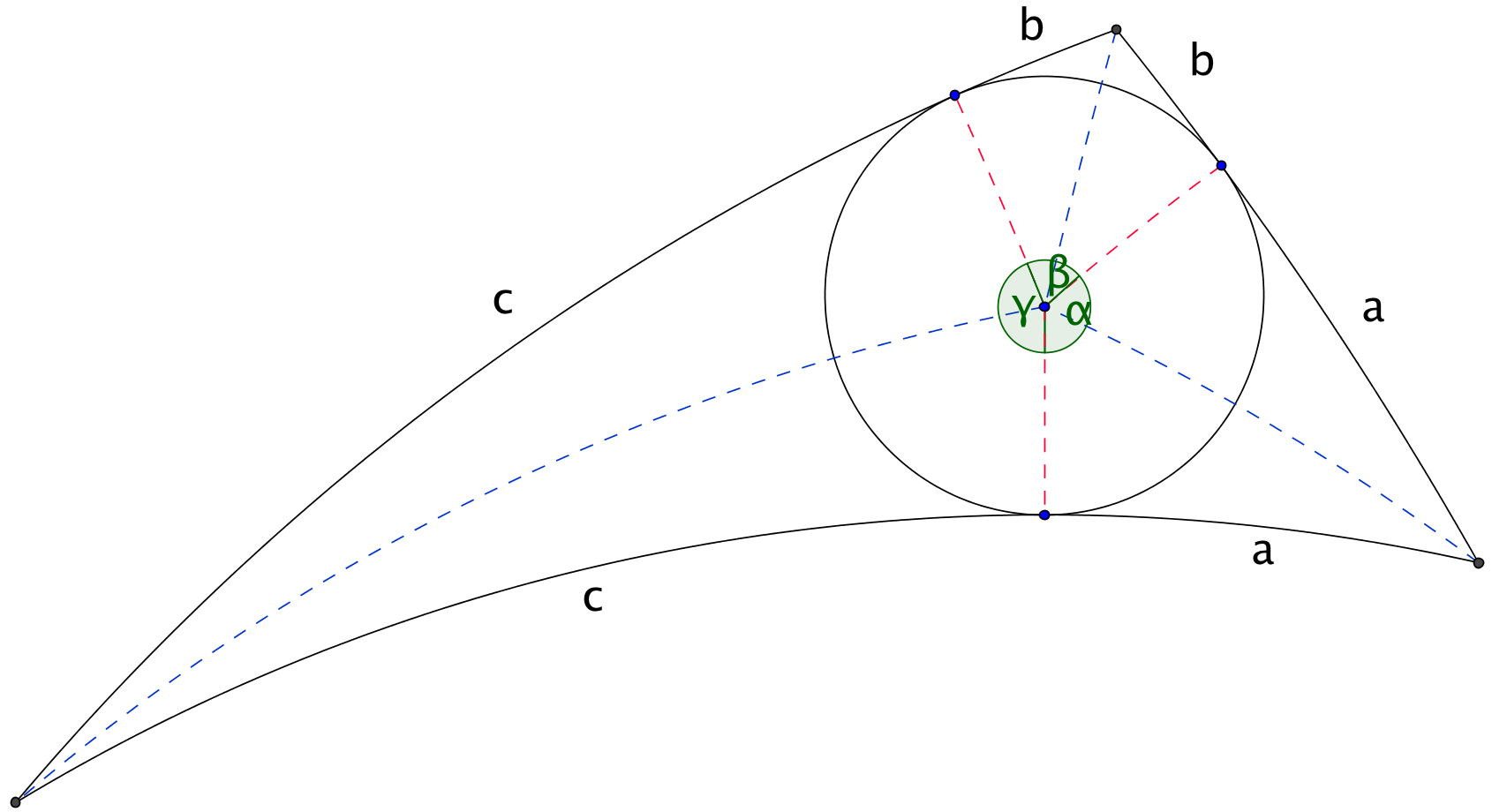}
\end{center}
\caption{Hyperbolic Triangle with Angle Bisectors, Incircle, and Inradii}
\label{H}
\end{figure}
The angle bisectors also bisect the central angles, so
\begin{align*}
\tanh(\a) = \rho\tan(\alpha/2) && \tanh(\b) = \rho\tan(\beta/2) && \tanh(\c) = \rho\tan(\gamma/2)
\end{align*}
where $\rho = \sinh(r)$. For convenience we define the notation
\begin{align*}
\tilde{u} \mathrel{\mathop:}= \tanh(u) \, \, \, \, \mbox{for all } u\in \R,
\end{align*}
so, similar to the Euclidean case, the Triple Tangent Identity implies
\begin{align*}
\tilde{\c} = \frac{\rho^2(\tilde{\a}+\tilde{\b})}{\tilde{\a}\tilde{\b}-\rho^2}.
\end{align*}
Thus using \ref{tanh} in Appendix \ref{A} shows the semiperimeter $s$ of $\Delta$ satisfies
\begin{align*}
\sigma \mathrel{\mathop:}= \tilde{s} = \tanh(\a+\b+\c) = \frac{ \tilde{\a}+\tilde{\b}+\tilde{\c}+\tilde{\a}\tilde{\b}\tilde{\c}}{1 + \tilde{\a}\tilde{\b}+\tilde{\b}\tilde{\c}+\tilde{\c}\tilde{\a}} = \cfrac{ \tilde{\a}+\tilde{\b}+(1+\tilde{\a}\tilde{\b})\cfrac{\rho^2(\tilde{\a}+\tilde{\b})}{\tilde{\a}\tilde{\b}-\rho^2}}{1 + \tilde{\a}\tilde{\b}+(\tilde{\b}+\tilde{\a})\cfrac{\rho^2(\tilde{\a}+\tilde{\b})}{\tilde{\a}\tilde{\b}-\rho^2}}
\end{align*}
or, equivalently,
\begin{align}\label{quartic}
 \sigma(\tilde{\a}^2\tilde{\b}^2+\tilde{\a}\tilde{\b} + \rho^2(\tilde{\a}^2 + \tilde{\a}\tilde{\b} + \tilde{\b}^2-1)) = (1+\rho^2)(\tilde{\a}^2\tilde{\b}+\tilde{\a}\tilde{\b}^2).
\end{align}
One can show that $\sigma \geq 3\sqrt{3}\rho\frac{1+\rho^2}{1+9\rho^2}$ with equality only for an equilateral triangle. Suppose $\Delta$ has side lengths $\ell_1,\ell_2, \ell_3$. Then, without loss of generality,
\begin{align*}
\a + \b= \ell_1  &&  \b+ \c = \ell_2 &&  \c+ \a = \ell_3,
\end{align*}
so
\begin{align*}
\tilde{\ell_1} = \frac{\tilde{\a}+\tilde{\b}}{1+\tilde{\a}\tilde{\b}} && \tilde{\ell_2} = \frac{\tilde{\b}+\tilde{\c}}{1+\tilde{\b}\tilde{\c}} &&  \tilde{\ell_3} =   \frac{\tilde{\c}+\tilde{\a}}{1+\tilde{\c}\tilde{\a}}.
\end{align*}
Therefore Eq. \ref{quartic} implies that we get a point
\begin{align}\label{point}
(\tilde{\a},\tilde{\b}) = \left(\frac{\sigma - \tilde{\ell}_2}{1-\sigma\tilde{\ell}_2}, \frac{\sigma - \tilde{\ell}_3}{1-\sigma\tilde{\ell}_3}\right) %= \left(\frac{\lambda_1^2\lambda_3^2 - \lambda_2^2}{\lambda_1^2\lambda_3^2 + \lambda_2^2}, \frac{\lambda_1^2\lambda_2^2 - \lambda_3^2}{\lambda_1^2\lambda_2^2 + \lambda_3^2} \right)
\end{align}
on the quartic curve
\begin{align}\label{heq}
Q_{\rho,\sigma}\colon \sigma(x^2y^2+xy + \rho^2(x^2 + xy + y^2-1)) = (1+\rho^2)(x^2y+xy^2)
\end{align}
where
\begin{align}\label{sigma}
\sigma = \frac{1 + \tilde{\ell}_1\tilde{\ell}_2 + \tilde{\ell}_2\tilde{\ell}_3+\tilde{\ell}_3\tilde{\ell}_1 + \sqrt{(1-\tilde{\ell}_1^2)(1-\tilde{\ell}_2^2)(1-\tilde{\ell}_3^2)}}{\tilde{\ell}_1 + \tilde{\ell}_2 + \tilde{\ell}_3+\tilde{\ell}_1\tilde{\ell}_2\tilde{\ell}_3},
\end{align}
\begin{align}\label{rho}
 \rho^2 = \left( \frac{(1-\sigma\tilde{\ell}_2)(1-\sigma\tilde{\ell}_3)}{(\sigma - \tilde{\ell_2})(\sigma - \tilde{\ell_3})} + \frac{(1-\sigma\tilde{\ell}_1)(1-\sigma\tilde{\ell}_3)}{(\sigma - \tilde{\ell_1})(\sigma - \tilde{\ell_3})} + \frac{(1-\sigma\tilde{\ell}_1)(1-\sigma\tilde{\ell}_2)}{(\sigma - \tilde{\ell_1})(\sigma - \tilde{\ell_2}))}\right)^{-1}.
\end{align}
Here the formula for $\sigma$ comes from the identity
\begin{align*}
 \frac{ \tilde{\ell}_1+\tilde{\ell}_2+\tilde{\ell}_3+\tilde{\ell}_1\tilde{\ell}_2\tilde{\ell}_3}{1 + \tilde{\ell}_1\tilde{\ell}_2+\tilde{\ell}_2\tilde{\ell}_3+\tilde{\ell}_3\tilde{\ell}_1} = \tanh(\ell_1+\ell_2 + \ell_3) = \tanh(2s) = \frac{2\sigma}{1+\sigma^2},
\end{align*}
and the formula for $\rho^2$ comes from
\begin{align*}
\frac{1}{\rho^2} = \frac{\tilde{\a}+\tilde{\b}+\tilde{\c}}{\tilde{\a}\tilde{\b}\tilde{\c}} = \frac{1}{\tilde{\b}\tilde{\c}} + \frac{1}{\tilde{\a}\tilde{\c}} + \frac{1}{\tilde{\a}\tilde{\b}}.
\end{align*}
The graph of $Q_{\rho,\sigma}$ in the affine real plane has five components (see Figure \ref{qrs}).
The point $(\tilde{\a},\tilde{\b})$ from $\Delta$ is, of course, on the component contained in $(0,1)\times(0,1)$. %, and, conversely, any point on this
%component comes from a triangle.
\begin{figure}
\begin{center}
\includegraphics[width=13 cm]{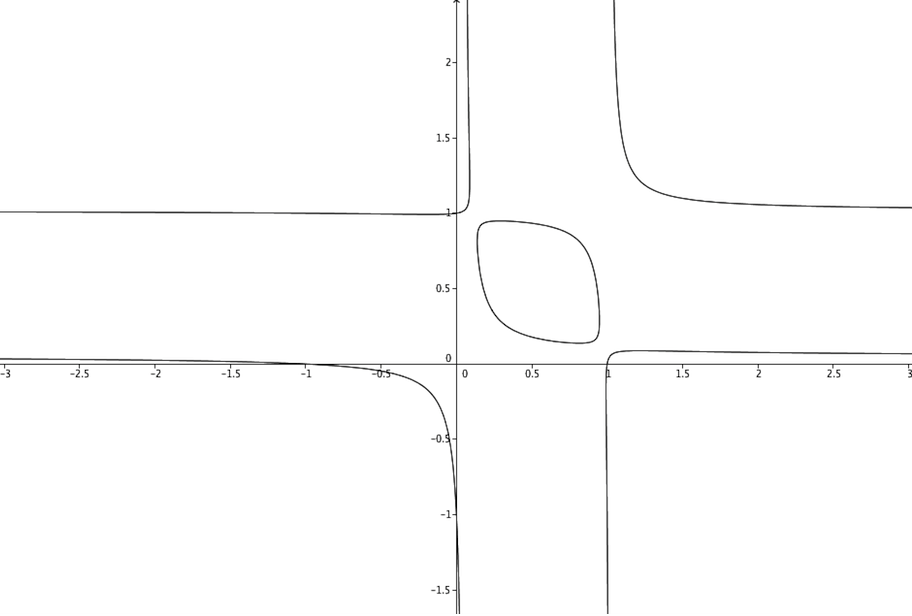}
\end{center}
\caption{Graph of $Q_{\rho,\sigma}$ in the affine real plane}
\label{qrs}
\end{figure}
As in the Euclidean case, we get up to six points on $Q_{\rho,\sigma}$, namely, $(\tilde{\a},\tilde{\b})$, $(\tilde{\b},\tilde{\c})$, $(\tilde{\c},\tilde{\a})$ and the transposes $(\tilde{\b},\tilde{\a})$, $(\tilde{\c},\tilde{\b})$, $(\tilde{\a},\tilde{\c})$. We homogenize Eq. \ref{heq} to view $Q_{\rho,\sigma}$ as a curve in $\P^2(\C)$ with projective coordinates $[x,y,z]$; the points at infinity $[1,0,0]$ and $[0,1,0]$ are the only singular points. These singular points both have multiplicity $2$, so the genus of $Q_{\rho, \sigma}$ is
\begin{align*}
\frac{(4-1)(4-2)}{2} - \frac{2(2-1)}{2} - \frac{2(2-1)}{2} = 3 - 1 - 1 = 1.
\end{align*}
There is a natural change of coordinates $(X,Y) = (xy, x+y)$ so that $\alpha^2-Y\alpha + X = (\alpha-x)(\alpha - y)$. Thus if $X$, $Y$ are in a subfield $F$ of $\C$, then $x,y$ will be in $F$ precisely when $Y^2-4X = Z^2$ for some $Z\in F$. This allows us to view our curve as the intersection of two quadric surfaces
\begin{align}\label{intersection}
H_{\rho,\sigma}\colon 
\left\{\begin{array}{l}
\sigma(X^2+X+\rho^2(Y^2-X-1))-(\rho^2+1)XY=0\\
4X-Y^2+Z^2=0
\end{array}\right.
\end{align}
\begin{figure}
\begin{center}
\includegraphics[width=16 cm]{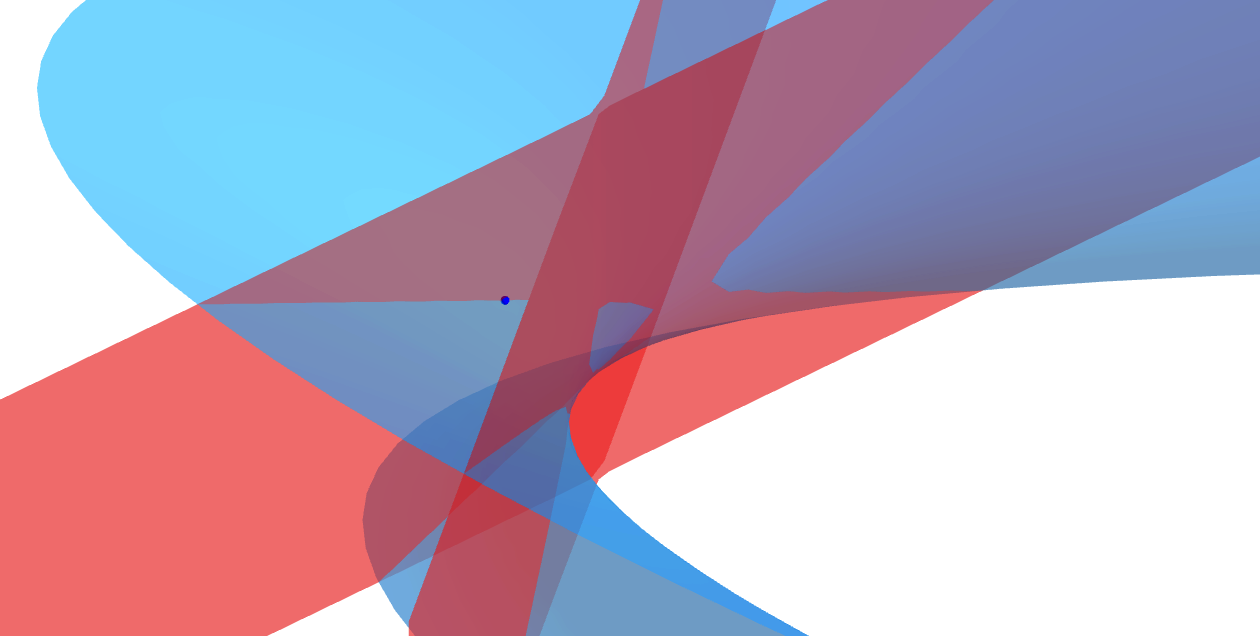}
\end{center}
\caption{Graph of $H_{\rho,\sigma}$ with base point $(-1,0,2)$}
\label{six}
\end{figure}
Thus the curve can be visualized, as in Figure \ref{six}, as the intersection of a hyperbolic cylinder (red) and a hyperbolic paraboloid (blue).
We homogenize Eq. \ref{intersection} and view $H_{\rho,\sigma}$ as a curve in $\P^3(\C)$ with projective coordinates $[X,Y,Z,T]$. There is a geometric way of combining points on $H_{\rho,\sigma}$ with respect to the point $\OO = [-1,0,2,1]$: for points $P$, $Q$, we take $P + Q$ to be the fourth intersection point of $H_{\rho,\sigma}$ with the plane spanned by $\OO$ and the secant/tangent line for $P$, $Q$. See \cite{Wang} for a more detailed exposition of curves arising from the intersections of two quadric surfaces.
%\begin{figure}
%\begin{center}
%\includegraphics[width=12 cm]{surfaces2}
%\end{center}
%\caption{Tangent Line and plane}
%\end{figure}

We say that $\Delta$ is \emph{rational} if $\tilde{\ell}_1, \tilde{\ell}_2, \tilde{\ell}_3\in\Q$. Note that $\Delta$ being
rational does not imply $\sigma\in\Q$ since, e.g., taking $\tilde{\ell}_1 = 1/2 = \tilde{\ell}_2$ and $\tilde{\ell}_3=1/3$ gives $\sigma = \frac{19+ 6\sqrt{2}}{17}\notin\Q$ by Equation \ref{sigma}.
However, if $\Delta$ is a rational right hyperbolic triangle with hypotenuse $\ell_3$, then recall that we have the following hyperbolic Pythagorean identity
\begin{align*}
\tilde{\ell}_3^2 = \tilde{\ell}_1^2+\tilde{\ell}_2^2-\tilde{\ell}_1^2\tilde{\ell}_2^2,
\end{align*}
so $1-\tilde{\ell}_3^2 = (1-\tilde{\ell}_1^2)(1-\tilde{\ell}_2^2)$ which combined with Equation \ref{sigma} implies
\begin{align*}
\sigma = \frac{1 + \tilde{\ell}_1\tilde{\ell}_2 + \tilde{\ell}_2\tilde{\ell}_3+\tilde{\ell}_3\tilde{\ell}_1 + 1-\tilde{\ell}_3^2}{\tilde{\ell}_1 + \tilde{\ell}_2 + \tilde{\ell}_3+\tilde{\ell}_1\tilde{\ell}_2\tilde{\ell}_3}\in\Q.
\end{align*}
In any case, it is clear from Eq.s \ref{point} and \ref{rho} that if $\Delta$ is rational and $\sigma\in\Q$, then we must also have $\tilde{a}, \tilde{b}, \tilde{c}, \rho^2\in\Q.$
\begin{thm}
Suppose $\sigma,\rho^2\in\Q$ with $\sigma \geq 3\sqrt{3}\rho\frac{1+\rho^2}{1+9\rho^2} > 0$. Then the assignment
\begin{align*}
(x_0,y_0) \mapsto \left(\tanh^{-1}\left(\frac{x_0+y_0}{1+x_0y_0}\right), \tanh^{-1}\left(\frac{\sigma-x_0}{1-\sigma x_0}\right), \tanh^{-1}\left(\frac{\sigma-y_0}{1-\sigma y_0}\right)\right)
\end{align*}
is a bijection from the rational points $(x_0,y_0)$ on $Q_{\rho,\sigma}$ having $1 > x_0, y_0 > 0$ to the sequences of side lengths of rational hyperbolic triangles having semiperimeter $s = \tanh^{-1}(\sigma)$ and inradius $r = \sinh^{-1}(\rho)$.
\end{thm}
\begin{proof}
We have already seen from the above construction that sequences of side lengths of rational hyperbolic triangles give rise to rational points on $Q_{\rho,\sigma}$ in $(0,1)\times (0,1)$; in particular, the map is onto. Conversely, suppose $(x_0,y_0)$ is a rational point on $Q_{\rho,\sigma}$ in $(0,1)\times (0,1)$. Then we claim
\begin{align*}
\left(\tanh^{-1}\left(\frac{x_0+y_0}{1+x_0y_0}\right), \tanh^{-1}\left(\frac{\sigma-x_0}{1-\sigma x_0}\right), \tanh^{-1}\left(\frac{\sigma-y_0}{1-\sigma y_0}\right)\right) \\
= (\tanh^{-1}(x_0) + \tanh^{-1}(y_0), s - \tanh^{-1}(x_0), s - \tanh^{-1}(y_0) )
\end{align*}
is a sequence of side lengths of a rational hyperbolic triangle. We have the inequalities
\begin{align*}
(\tanh^{-1}(x_0)+\tanh^{-1}(y_0) ) + (s-\tanh^{-1}(x_0)) = s + \tanh^{-1}(y_0) > s - \tanh^{-1}(y_0) \\
(\tanh^{-1}(x_0)+\tanh^{-1}(y_0) ) + (s-\tanh^{-1}(y_0)) = s + \tanh^{-1}(x_0) > s-\tanh^{-1}(x_0). 
\end{align*}
Also,
\begin{align*}
0>(x_0-1)(1-y_0)  = x_0+y_0-(x_0y_0+1),
\end{align*}
so $(x_0y_0+1) > (x_0+y_0)^2/(1+x_0y_0)$ and
\begin{align*}
1<\frac{(1+\rho^2)x_0y_0}{\left(x_0y_0 +\rho^2 \left( \frac{(x_0+y_0)^2}{1+x_0y_0} -1\right)\right)} = \frac{\sigma(1+x_0y_0)}{x_0+y_0}.
\end{align*}
Thus
\begin{align*}
&\hspace{0.2 in} (s-\tanh^{-1}(x_0)) + (s-\tanh^{-1}(y_0)) \\
&= 2\tanh^{-1}(\sigma)-(\tanh^{-1}(x_0) +\tanh^{-1}(y_0)) \\
&>2\tanh^{-1}\left(\frac{x_0+y_0}{1+x_0y_0}\right) -(\tanh^{-1}(x_0) +\tanh^{-1}(y_0))\\
&=2(\tanh^{-1}(x_0) +\tanh^{-1}(y_0)) - (\tanh^{-1}(x_0) +\tanh^{-1}(y_0)) \\
&=\tanh^{-1}(x_0) +\tanh^{-1}(y_0).
\end{align*}
It remains only to note that $(\tanh^{-1}(x_0)+\tanh^{-1}(y_0), s-\tanh^{-1}(x_0), s-\tanh^{-1}(y_0)) =(\tanh^{-1}(x'_0)+\tanh^{-1}(y'_0), s-\tanh^{-1}(x'_0), s-\tanh^{-1}(y'_0)) \Rightarrow (x_0, y_0) =(x_0', y_0')$.
\end{proof}
\begin{exam}
There is a construction, discussed in \cite{Hart2} for instance, which allows one to write down rational\footnote{As mentioned in the introduction, the notion of a rational hyperbolic triangle in Hartshorne and van Luijk's article \cite{Hart2} is a stronger notion than ours; namely, they require that the side lengths are logs of rational numbers, whereas we only require that the hyperbolic tangents of the side length are rational numbers or, equivalently, that the side lengths are the logs of square roots of rational numbers.} hyperbolic Pythagorean triples. For example, there is a hyperbolic right triangle with side lengths $\tanh^{-1}(672/697)$, $\tanh^{-1}(104/185)$, $\tanh^{-1}(40/41)$. We have $\sigma = 312/317$ and $\rho^2 = 112^2/242201$ in this case by Equations \ref{sigma} and \ref{rho}, so by Equation \ref{point} we get a rational point
\begin{align*}
(\tilde{\a},\tilde{\b}) = \left(\frac{1456}{1541}, \frac{112}{517} \right)
\end{align*}
on the rational quartic $Q_{\rho, \sigma}$. This gives us a rational point
\begin{align*}
P=\left(\frac{163072}{796697}, \frac{925344}{796697}, \frac{580160}{796697}\right)
\end{align*}
in the affine patch $(X,Y,Z) = [X,Y,Z,1]$ of $H_{\rho,\sigma}$. In order to compute $P + P$ on $H_{\rho,\sigma}$, we first compute the tangent line at $P$ by intersecting the tangent planes at $P$ on the two surfaces defining $H_{\rho,\sigma}$. This tangent line at $P$ is parallel to a vector $v$ obtained by taking a cross product of normal vectors of the tangent planes. We find
\begin{align*}
v \times (P-\OO) =  \left(315005821528688640, -\frac{1615149902619671040}{11}, \frac{1807175612011392000}{11}\right),
\end{align*}
which is the normal vector of the plane determined by the point $\OO = (-1,0,2)$ and the tangent line for $P$. We find the fourth intersection of this plane with $H_{\rho,\sigma}$ to be
\begin{align*}
P + P = \left(\frac{157101469162847924}{6671549185609843471}, -\frac{3620500406298490680}{6671549185609843471}, \frac{2985897265044714172}{6671549185609843471}\right).
\end{align*}
This gives us a rational point on $Q_{\rho,\sigma}$, but does not correspond to a triangle, of course, since the point has negative $y$-coordinate. However, we can now compute $P + (P + P)$. We compute $((P + P)-\OO) \times (P-\OO) =$
\begin{align*}
\left(\frac{13252388009067818908542000}{5315203221527805463815287},
-\frac{274681775539499477051700}{483200292866164133074117},
\frac{208376488637078243567400}{113089430245272456676921}\right),
\end{align*}
which is the normal vector for the plane determined by the point $\OO$ and the secant line for $P$, $P + P$. The fourth intersection point of this plane with $H_{\rho,\sigma}$ has coordinates
\begin{align*}
& \left(\frac{5263691075333761370098794452493746497879860068608}{16384416048645387501313685205281020081739448108365},\right. \\
& \hspace{0.17 in} \frac{21085154660473767875635316117377007639096311317792}{16384416048645387501313685205281020081739448108365}, \\
& \hspace{0.15 in}  \left. \frac{9980667760067897618538531202408325397579645579328}{16384416048645387501313685205281020081739448108365}\right).
\end{align*}
This point does indeed give a point on $Q_{\rho,\sigma}$ in the first quadrant with coordinates smaller than $1$:
\begin{align*}
(x_0,y_0) = \left(\frac{2072869433189638375660592}{2186502887201310556520693}, \frac{2539325917520154646338224}{7493434444816664924429305} \right)
\end{align*}
This point corresponds to a non-right rational hyperbolic triangle with semiperimeter $s = \tanh^{-1}(312/317)$ and inradius $r = \sinh^{-1}(112/\sqrt{242201})$, namely, the triangle with side lengths
\begin{align*}
&\tanh^{-1}\left(\frac{4938503954557916283489312}{5070357052721862942058853} \right), \\
&\tanh^{-1}\left(\frac{25089290485693528550048552}{46386152087648273210954977} \right),\\
&\tanh^{-1}\left(\frac{1532985230928910433532726152}{1583149032740594531386563797}\right).
\end{align*}
This example would have been very difficult to find by hand, i.e., without exploiting the elliptic curve structure.
\end{exam}

\appendix

\section{Two Trigonometric Identities}\label{A}

The key identity we used in deriving our plane curves came from relations
among the tangents of half the central angles.
\begin{thm}[Triple Tangent Identity]\label{trip}
Suppose $\theta_1 + \theta_2 + \theta_3 = \pi$. Then

\begin{align*}
\tan(\theta_1) + \tan(\theta_2) + \tan(\theta_3) = \tan(\theta_1)\tan(\theta_2)\tan(\theta_3).
\end{align*}
\end{thm}

\begin{proof}
Using oddness, $\pi$-periodicity, and the addition law, we get
\begin{align*}
-\tan(\theta_3) = \tan(\pi - \theta_3) = \tan(\theta_1+\theta_2) = \frac{\tan(\theta_1)+\tan(\theta_2)}{1-\tan(\theta_1)\tan(\theta_2)}.
\end{align*}
\end{proof}

In the hyperbolic case, the trigonometric formulas for right triangles involved hyperbolic tangents
of side lengths, so we needed to make use of an iterated addition law.

\begin{thm}\label{tanh}
For $A,B,C \in\R$, we have
\begin{align*}
\tanh(\tanh^{-1}(A) + \tanh^{-1}(B) + \tanh^{-1}(C)) = \frac{A+B +C + ABC}{1+AB+BC+CA}.
\end{align*}
\end{thm}
\begin{proof}
Use the definition of the hyperbolic tangent function $\tanh(z) = \displaystyle{\frac{e^{z}-e^{-z}}{e^{z}+e^{-z}} =  \frac{e^{2z}-1}{e^{2z}+1} }$ to get
\begin{align*}
\frac{\tanh(A)+\tanh(B)}{1+\tanh(A)\tanh(B)} %&= \cfrac{\cfrac{e^{2A}-1}{e^{2A}+1} +  \cfrac{e^{2B}-1}{e^{2B}+1}}{1 + \cfrac{e^{2A}-1}{e^{2A}+1} \cdot \cfrac{e^{2B}-1}{e^{2B}+1}} \\
&= \frac{(e^{2A}-1)(e^{2B}+1) +  (e^{2B}-1)(e^{2A}+1)}{(e^{2A}+1)(e^{2B}+1) + (e^{2A}-1)(e^{2B}-1)} \\
&=  \frac{2e^{2(A+B)}  -2}{2e^{2(A+B)} +2} = \tanh(A + B),
\end{align*}
so
\begin{align*}
\tanh(\tanh^{-1}(A) + \tanh^{-1}(B)) = \frac{A+B}{1+AB}.
\end{align*}
Thus
\begin{align*}
\tanh(\tanh^{-1}(A) + \tanh^{-1}(B) + \tanh^{-1}(C)) &= \frac{\tanh(\tanh^{-1}(A) + \tanh^{-1}(B)) +C }{1+ \tanh(\tanh^{-1}(A) + \tanh^{-1}(B))C} \\
&=  \frac{A+B +C(1+AB) }{1+AB+(A+B)C}.
\end{align*}
\end{proof}

%\begin{align*}
%\tan(A+B) = i \frac{\tanh(A/i) + \tanh(B/i)}{1+ \tanh(A/i) \tanh(B/i)} = \frac{\tan(A)+\tan(B)}{1 - \tan(A)\tan(B)}
%\end{align*}
%This addition law combined with oddness and $\pi$-periodicity gives the following identity, which is central to our derivations.

\bibliographystyle{amsalpha}
\bibliography{References}

\providecommand{\bysame}{\leavevmode\hbox to3em{\hrulefill}\thinspace}
\providecommand{\MR}{\relax\ifhmode\unskip\space\fi MR }
% \MRhref is called by the amsart/book/proc definition of \MR.
\providecommand{\MRhref}[2]{%
  \href{http://www.ams.org/mathscinet-getitem?mr=#1}{#2}
}
\providecommand{\href}[2]{#2}
\begin{thebibliography}{GJW03}

\bibitem[CG]{Camp}
Garikai Campbell and Edray~Herber Goins, \emph{Heron triangles, diophantine
  problems and elliptic curves}, preprint.

\bibitem[GJW03]{Wang}
Ronald Goldman, Barry Joe, and Wenping Wang, \emph{Computing quadric surface
  intersections based on an analysis of plane cubic curves}, Graphic Models
  \textbf{64} (2003), 335--367.

\bibitem[GM06]{Goin}
Edray~Herber Goins and Davin Maddox, \emph{Heron triangles via elliptic
  curves}, Rocky Mountain J. Math. \textbf{36} (2006), no.~5, 1511--1526.

\bibitem[HvL08]{Hart2}
Robin Hartshorne and Ronald van Luijk, \emph{Non-{E}uclidean {P}ythagorean
  triples, a problem of {E}uler, and rational points on {K3} surfaces}, The
  Mathematical Intelligencer \textbf{30} (2008), 4--10.

\bibitem[Kob93]{Kobl}
Neal Koblitz, \emph{Introduction to elliptic curves and modular forms}, 2nd
  ed., Springer, 1993.

\bibitem[RSW08]{Rose}
Steven Rosenberg, Michael Spillane, and Daniel Wulf, \emph{Heron triangles and
  moduli spaces}, The Mathematics Teacher \textbf{101} (2008), 656--663.

\end{thebibliography}

\end{document}